\documentclass[11pt]{article}

\usepackage{amsmath,amsfonts,amssymb,amsthm}
\usepackage{fullpage,subfiles}

\usepackage{graphicx}
\usepackage[colorlinks=true, allcolors=blue]{hyperref}

\newtheorem{theorem}{Theorem}
\newtheorem{corollary}[theorem]{Corollary}
\newtheorem{lemma}[theorem]{Lemma}

\theoremstyle{definition}

\newtheorem{example}[theorem]{Example}

\newcommand{\N}{\mathbb{N}}

\newcommand{\expect}[2]{\mathbb{E}_{#1} \left\{ {#2} \right\}}

\DeclareMathOperator{\diag}{diag}
\DeclareMathOperator{\perm}{perm}

\newcommand{\AND}{\quad \text{and}\quad}

\newcommand{\R}{\mathbb{R}}

\newcommand{\mydet}[1]{\mathrm{det}\left[ {#1} \right]}

\newcommand{\trans}{T}
\newcommand{\rr}{r}

\newcommand{\inv}{\dagger}

\newcommand{\MM}{\mathcal{M}}

\renewcommand{\ss}{n_1}
\renewcommand{\tt}{n_2}
\newcommand{\kk}{m}

\author{Adam W. Marcus\thanks{Research supported by NSF CAREER Grant 
DMS-1552520.}
\\ \'Ecole Polytechnique F\'ed\'erale de Lausanne
}
\title{A class of multivariate polynomial convolutions (and applications)}

\begin{document}

\maketitle

\begin{abstract}
We prove two ``master'' convolution theorems for multivariate determinantal 
polynomials.
The methods used include basic properties of what we call a 
``minor-orthogonal'' ensemble as well as properties of the mixed discriminant 
of matrices.
We also give applications, including a rederivation of a result of Barvinok on 
computing the permanent of a low rank matrix and a polynomial convolution 
corresponding to the unitarily invariant addition of generalized singular 
values. 
\end{abstract}


\section{Introduction}\label{sec:intro}

The primary goal of this paper is to prove two ``master'' convolution 
theorems 
for determinantal polynomials.
One unavoidable fact of general matrices is that they are two dimensional 
objects, containing both a width and a length.
While the dimensions of a matrix tend not to appear explicitly in basic linear 
algebra formulas, they seem to have a more direct role in the context of 
random matrices.
One obvious example of this is the {\em Wishart ensemble}: let $W = X  X^*$ 
where 
$X$ is an $n \times m$ random matrix with independent real Gaussian entries.
Even though $W$ is, itself, an $n \times n$ matrix (no $m$ involved), the joint 
eigenvalue 
has the form 
\[
\mu_W(\lambda_1, \dots \lambda_n) 
\propto e^{-1/2 \sum_i \lambda_i} 
\prod_i \lambda_i^{(m-n-1)/2} 
\prod_{i < j} |\lambda_i 
- \lambda_j|
\]
where the value of $m$ is considered to be a measure of ``degrees of freedom'' 
\cite{forrester}.
In this note, we will refer to the value of $m$ in this example as a 
{\em local} dimension (since the $m$ disappears after the product is taken) and 
to $n$ as a {\em global} dimension (since the product matrix still has $n$ as a 
dimension).
Similar to the case of the Wishart ensemble, polynomial convolutions will 
depend on both local and global parameters, and 
so the aim will be to find methods for computing both in the most general case 
possible.

To state the results explicitly, we first introduce some notations: let 
$\MM_{m, 
n}$ denote the collection of $m \times n$ matrices\footnote{All of 
results in this paper will hold true for any base field.
This is not the case for the random matrix ensembles that we will mention at 
various points, and so we will state the base field specifically in those 
discussions.}.
We will use the standard notation for multivariate polynomials: for $\alpha \in 
\N^k$, we write
\[
x^\alpha := \prod_{i=1}^k x_i^{\alpha_i}
\AND
\alpha! := \prod_{i=1}^k \alpha_i!
\]
Given two degree $n$ homogeneous polynomials
\[
p(x_1, \dots, x_k) = \sum_{\alpha \in \N^k} p_\alpha x^{\alpha}
\AND
q(x_1, \dots, x_k) = \sum_{\alpha \in \N^k} q_\alpha x^{\alpha},
\]
we define the $\star$-convolution of $p$ and $q$ to be\footnote{Note that this is {\em not} the same as the Schur--Hadamard convolution, which is defined by
\[
x^{\alpha} \bullet x^{\beta} = \frac{\delta_{\{\alpha = \beta\}}}{\binom{\eta}{\alpha}}
\]
where $\eta \in \N^k$ contains the maximum degree of each variable in $p$ and $q$ (see \cite{bb}).}

\begin{equation}\label{eq:conv}
[p \star q](x_1, \dots, x_k) = \frac{1}{n!}\sum_{\alpha \in \N^k} p_\alpha q_\alpha \alpha! x^{\alpha}.
\end{equation}

We also define an operator on multivariate polynomials that operates on pairs of variables (all other variables considered fixed). 
For integers $i, j, m$, we define
\begin{equation}\label{eq:L}
L_m^{x, y} [ x^i y^j ] = 
\begin{cases}
\frac{(m-i)!(m-j)!}{m!(m-i-j)!} x^i y^j & \text{for ${i + j \leq m}$} \\
0 & otherwise
\end{cases}
\end{equation}
and extend linearly to generic multivariate polynomials.
Our first major theorem shows the effect of averaging over a local dimension:
\begin{theorem}[Local]\label{thm:local}
For integers $d,m$ and variables $x,y$ let
\begin{itemize}
\item $R \in \MM_{m, m}$ be a uniformly distributed signed 
permutation matrix
\item $A_1, A_2 \in \MM_{d, m}$ and $B_1, B_2 \in \MM_{m, d}$ and $U \in 
\MM_{d, d}$ be matrices that are independent from $R$ and do not contain the 
variables $x$ and $y$ (but could contain other variables).
\end{itemize}
Then 
\[
\expect{R}{ \mydet{U + (x A_1 + y A_2 R)(B_1 + R^T B_2)}}
 = L_m^{x, y} 
\left\{ \mydet{U + x A_1 B_1 + y A_2 B_2} \right \} 
\]
\end{theorem}

The second main theorem then shows the effect of averaging over a global dimension:
\begin{theorem}[Global]\label{thm:global}
For integers $n, d$, let $A_1, \dots, A_n, B_1, \dots, B_n \in \MM_{d,d}$ and 
set 
\[
p(x_1, \dots, x_n) = \mydet{ \sum_i x_i A_i }
\AND
q(x_1, \dots, x_n) = \mydet{ \sum_i x_i B_i }.
\]
If $Q \in \MM_{d, d}$ is a uniformly distributed signed permutation matrix, 
then 
\[
\expect{Q}{ \mydet{\sum_i x_i A_i Q B_i Q^T } } = [p \star q](x_1, \dots, x_n)
\]
\end{theorem}
These theorems can then be used iteratively to compute more complicated 
convolutions (we give examples of this in 
Sections~\ref{sec:nonHermitian}~and~\ref{sec:gsvd}).

The paper will proceed by first reviewing some of the basic combinatorial and 
linear algebraic tools that we will need (Section~\ref{sec:prelims}).
We will then introduce a type of random matrix ensemble which we call {\em 
minor-orthogonal} and prove some basic properties 
(Section~\ref{sec:minor-orthogonal}).
Among these properties will be the fact that a uniformly distributed signed 
permutation matrix is minor-orthogonal (Lemma~\ref{lem:sps})
In Section~\ref{sec:local}, we will give a proof of Theorem~\ref{thm:local} in 
the more general context of minor-orthogonal matrices.
Unfortunately, we are not able to prove Theorem~\ref{thm:global} in similar 
generality.
Instead, we present a proof of Theorem~\ref{thm:global} specific to signed 
permutation matrices in Section~\ref{sec:global}.
We give some examples of applications (reproducing known results) in 
Section~\ref{sec:apps} and then give the main application (the introduction of 
an additive convolution for generalized singular values) in 
Section~\ref{sec:gsvd}.
Finally, we discuss some open problems in Section~\ref{sec:conc}.

\section{The tools}

We start by giving the definitions and constructs that we will use.

\subsection{General}

For a statement $S$, we will use the Dirac delta function
\[
\delta_{ \{ S \} } = 
\begin{cases}
1 & \text{if $S$ is true} \\
0 & \text{if $S$ is false}
\end{cases}.
\]

We write $[n]$ to denote the set $\{ 1, \dots, n \}$ and for a set $S$, we write $\binom{S}{k}$ to denote the collection of subsets of $S$ that have exactly $k$ elements.
For example, 
\[
\binom{[4]}{2} = \big\{ \{ 1, 2 \}, \{ 1, 3 \}, \{ 1, 4 \}, \{ 2, 3 \}, \{ 2, 4 \}, \{ 3, 4 \} \big\}
\]

When our sets contain integers (which they always will), we will consider the set to be ordered from smallest to largest.
Hence, for example, if $S$ contains the elements $\{ 2, 5, 3 \}$, then we will write 
\[
S = \{ s_1, s_2, s_3 \} 
\quad
\text{where}
\quad
s_1 = 2, s_2 = 3, s_3 = 5.
\]
Now let $S = \{ s_1, \dots, s_k \} \in \binom{[n]}{k}$.
For a set $W \in \binom{[k]}{j}$ with $j \leq k$, we will write
\[
W(S) = \{ s_i : i \in W \}.
\]
Lastly, for a set of integers $S$, we will write 
\[
\| S \|_1 = \sum_{s \in S} s
\]
and note that (as is easy to check)
\[
(-1)^{\| S + T \|_1} = (-1)^{\| S\|_1 + \|T \|_1}.
\]

\begin{example}
For $W = \{ 1, 3 \}$ and $S = \{ 2, 4, 5 \}$ we have
\[
W(S) = \{ 2, 5 \}
\AND
\|W\|_1 = 1 + 3 = 4
\AND
\|S\|_1 = 2 + 4 + 5 = 11.
\]
\end{example}

\subsection{Matrices}
Given a matrix $A \in \MM_{n, n}$ and sets $S \in \binom{[n]}{k}$ and $T \in 
\binom{[m]}{k}$, we will write the {\em $(S, T)$-minor of $A$} as
\[
[A]_{S, T} = \mydet{ \{ a_{i, j} \}_{i \in S, j \in T}}
\]
By definition, we will set $[A]_{\varnothing, \varnothing} = 1$.
There are well-known formulas for the minor of a product of matrices 
\cite{horn} 
as well as the minor of a sum of matrices \cite{mm}:
\begin{theorem}
For integers $m, n, p, k$ and matrices $A \in \MM_{m, n}$ and $B \in \MM_{n, 
p}$, we have
\begin{equation}\label{eq:mult}
[A B]_{S, T} = \sum_{|U| \in \binom{[n]}{k}}
[A]_{S, U} [B]_{U, T}.
\end{equation}
for any sets $S \in \binom{[m]}{k}$ and $T \in \binom{[p]}{k}$.
\end{theorem}

\begin{theorem}
For integers $n, k$ and matrices $A, B \in \MM_{n, n}$, we have
\begin{equation}\label{eq:add}
[A + B]_{S, T} = \sum_i \sum_{U, V \in \binom{[k]}{i}} 
(-1)^{\| U(S) + V(T) \|_1} 
[A]_{U(S), V(T)} 
[B]_{\overline{U}(S), \overline{V}(T)}
\end{equation}
for any sets $S, T \in \binom{[n]}{k}$.
\end{theorem}

We will also make use of the {\em mixed discriminant}: for an integer $n$, let 
$X_1, \dots, X_n \in \MM_{n, n}$.
The {\em mixed discriminant} of these matrices is then defined as
\[
D(X_1, \dots, X_n) = \frac{1}{n!}\frac{\partial^n}{(\partial_{t_1}) \dots (\partial_{t_n})}~ \mydet{ \sum_i t_i X_i}
\] 
Note that $\mydet{ \sum_i t_i X_i}$ is a degree $n$ homogeneous polynomial, so the derivative is precisely the coefficient of $t^\alpha$ where $\alpha = (1, 1, 1, \dots, 1)$.

The mixed discriminant has the following well-known properties (all following directly from the definition):
\begin{lemma}\label{lem:md}
Let $X_1, \dots, X_n, Y \in \MM_{n, n}$ and let $\{u_i \}_{i=1}^n, \{ v_i 
\}_{i=1}^n$ be vectors of length $n$. 
Then
\begin{enumerate}
\item $D(a X_1 + b Y, X_2, \dots, X_n) = a D(X_1, X_2, \dots, X_n) + b D(Y, X_2, \dots, X_n)$ for all scalars $a, b$
\item $D(X_1, X_2, \dots, X_n) = D(X_{\pi(1)}, X_{\pi(2)}, \dots, X_{\pi(n)})$ for all permutations $\pi$
\item 
$D(X_1 Y, X_2 Y, \dots, X_n Y) 
= D(Y X_1, Y X_2, \dots, Y X_n)
= \mydet{Y} D(X_1, X_2, \dots, X_n)$ 
\item $D(u_1 v_1^T, u_2 v_2^T, \dots, u_n v_n^T) = \mydet{ u_1~u_2~\dots~u_n} \mydet{v_1~v_2~\dots~v_n}$
\end{enumerate}
\end{lemma}
Note that properties {\em 1.} and {\em 2.} combine to show that the mixed 
discriminant is multilinear (that is, it is a linear function with respect to 
each of its inputs).

We will leave the discussion of generalized singular values to 
Section~\ref{sec:gsvd}.

\label{sec:prelims}
\section{Minor-orthogonality}

\newcommand{\mydel}[1]{\delta_{ \{ #1 \} } }
\newcommand{\minor}[3]{[#1]_{#2, #3}}

\newcommand{\mm}{m}
\newcommand{\nn}{n}

\newcommand{\cc}{c}
\newcommand{\dd}{d}

\newcommand{\bin}{ \frac{1}{\binom{ \max \{ \mm, \nn \} }{k}}}

We will say that a random matrix $R \in \MM_{\mm, \nn}$ is {\em 
minor-orthogonal} if for all integers $k, \ell \leq \min \{ \mm, \nn \}$ and 
all sets $S, T, U, V$ with $|S| = |T| = k$ and $|U| = |V| = \ell$, we have
\[
\expect{R}{\minor{R}{S}{T}\minor{R^T}{U}{V}} 
= \bin \mydel{S = V}\mydel{T = U}.
\]
Given a minor-orthogonal ensemble $R$ it is easy to see from the definition that
\begin{enumerate} 
\item $R^T$ is minor orthogonal
\item any submatrix that preserves the largest dimension of $R$ is minor 
orthogonal
\end{enumerate}

\begin{lemma}\label{lem:rot}
If $R$ is minor-orthogonal and $Q$ is a fixed matrix for which $QQ^T = I$, then 
$QR$ is minor-orthogonal.
\end{lemma}
\begin{proof}
For any sets $S, T$ with $|S| = |T| = k$, we have
\[
\minor{QR}{S}{T} = \sum_{|W| = k} \minor{Q}{S}{U} \minor{R}{U}{T}
\]
so for $|U| = |V| = \ell$, we have
\begin{align*}
\expect{R}{\minor{QR}{S}{T}\minor{(QR)^T}{U}{V}} 
&= 
\expect{R}{ 
\sum_{|W| = k}\sum_{|Z| = \ell} 
\minor{Q}{S}{W} \minor{R}{W}{T} 
\minor{R^T}{U}{Z} \minor{Q^T}{Z}{V} 
}
\\&= 
\sum_{|W| = k}\sum_{|Z| = \ell} \minor{Q}{S}{W} \minor{Q^T}{Z}{V} 
\bin \mydel{W = Z} \mydel{T = U}
\\&= 
\sum_{|W| = k} \bin \minor{Q}{S}{W} 
\minor{Q^T}{W}{V} \mydel{T = U}
\\&=
\bin
\mydel{S = V}\mydel{T = U}.
\end{align*}
where the last line comes from the fact that $\minor{I}{S}{V} = \mydel{S = V}$.
\end{proof}

\begin{lemma}\label{lem:sps}
The collection of $\nn \times \nn$ signed permutation matrices (under the 
uniform distribution) is minor-orthogonal.
\end{lemma}
\begin{proof}
We can write a uniformly random signed permutation matrix $Q$ as $Q = E_\chi 
P_\pi$ where $P_\pi$ is a uniformly random permutation matrix and $E_\chi$ is a 
uniformly random diagonal matrix with $\{ \pm 1 \}$ on the diagonal (and the 
two are independent).
Hence for $|S| = |T| = k$ and $|U| = |V| = \ell$, we have
\begin{align*}
\expect{Q}{\minor{Q}{S}{T}\minor{Q^T}{U}{V}} 
&=
\expect{\chi, \pi}{\minor{E_\chi P_\pi}{S}{T} \minor{ P_\pi^T E_\chi}{U}{V}} 
\\&= 
\sum_{|W| = k}\sum_{|Z| = \ell} \expect{\chi, \pi}{\minor{E_\chi}{S}{W} 
\minor{P_\pi}{W}{T} \minor{ P_\pi^T }{U}{Z} \minor{E_\chi}{Z}{V}}.
\\&= 
\expect{\chi, \pi}{\minor{E_\chi}{S}{S} \minor{P_\pi}{S}{T} \minor{ P_\pi^T 
}{U}{V} \minor{E_\chi}{V}{V}}
\\&=
\expect{\chi}{ \prod_{i \in S} \chi_i \prod_{j \in V} \chi_j}
\expect{\pi}{\minor{P_\pi}{S}{T} \minor{ P_\pi^T }{U}{V}}.
\end{align*}
where the penultimate line uses the fact that a diagonal matrix $X$ satisfies 
$\minor{X}{A}{B} = 0$ whenever $A \neq B$.
Now the $\chi_i$ are uniformly distributed $\{ \pm 1 \}$ random variables, so 
\[
\expect{\chi}{ \prod_{i \in S} \chi_i \prod_{j \in V} \chi_j} = \mydel{S = V}
\]
and so we have
\begin{align*}
\expect{Q}{\minor{Q}{S}{T}\minor{Q^T}{U}{V}} 
&= 
\expect{\pi}{\minor{P_\pi}{S}{T} \minor{ P_\pi^T }{U}{V}} \mydel{S = V}
\\&=
\expect{\pi}{\minor{P_\pi}{S}{T} \minor{ P_\pi }{S}{U}} \mydel{S = V}
\end{align*}
Furthermore, $\minor{P_\pi}{S}{T} = 0$ except when $T = \pi(S)$, so in order 
for both $\minor{P_\pi}{S}{T}$ and $\minor{ P_\pi }{S}{U}$ to be nonzero 
simultaneously requires $U = T$.
In the case that $U = T = \pi(S)$, $\minor{P_\pi}{S}{T} = \pm 1$, and so we have
\begin{align*}
\expect{Q}{\minor{Q}{S}{T}\minor{Q^T}{U}{V}} 
&= 
\expect{\pi}{\minor{P_\pi}{S}{T}^2} \mydel{S = V} \mydel{T = U}
\\&= 
\expect{\pi}{\mydel{\pi(S) = T}} \mydel{S = V} \mydel{T = U}
\end{align*}

But it is an easy exercise to check that the probability that a permutation 
length $\nn$ maps a set $S$ to a set $T$ with $|S| = |T| = k$ is 
\[
\frac{k!(\nn-k)!}{\nn!} = \frac{1}{\binom{\nn}{k}}
\]
and so for $|S| = |T| = k$, we have
\[
\expect{Q}{\minor{Q}{S}{T}\minor{Q^T}{U}{V}}  
= \frac{1}{\binom{\nn}{k}}  \mydel{S = V} \mydel{T = U}
\]
as required.
\end{proof}

As mentioned at the end of Section~\ref{sec:intro}, Lemma~\ref{lem:sps} can be 
extended to show that any suitably symmetric group containing the signed 
permutation matrices is minor-orthogonal.
An example of this is given in Corollary~\ref{cor:mo}:

\begin{corollary}\label{cor:mo}
The collection of $\nn \times \nn$ orthogonal matrices (under the 
Haar measure) is minor-orthogonal.
\end{corollary}
\begin{proof}
Let $R$ be a Haar distributed random orthogonal matrix.
By definition, $RQ$ is also Haar distributed for any fixed orthogonal matrix 
$Q$, and so (in particular) this holds when $Q$ is a signed permutation matrix. 
Hence by Lemma~\ref{lem:rot}
\[
\expect{R}{\minor{R}{S}{T}\minor{R^T}{U}{V}} 
=
\expect{R}{\minor{R Q}{S}{T}\minor{(RQ)^T}{U}{V}} 
\]
and so 
\[
\expect{R}{\minor{R}{S}{T}\minor{R^T}{U}{V}} 
=
\expect{R, Q}{\minor{R}{S}{T}\minor{R^T}{U}{V}} 
=
\expect{R, Q}{\minor{RQ}{S}{T}\minor{(RQ)^T}{U}{V}}
\]
where we are now considering $Q$ to be drawn uniformly and independently from 
the 
collection of signed permutation matrices.
By Lemma~\ref{lem:sps}, $Q$ is minor-orthogonal and so, for fixed $R$,  
Lemma~\ref{lem:rot} implies that $RQ$ is minor-orthogonal. 
So 
\[
\expect{R}{\expect{Q}{\minor{RQ}{S}{T}\minor{(RQ)^T}{U}{V}}}
= 
\expect{R}{ \frac{1}{\binom{n}{k}} \mydel{S = V}\mydel{T = U} }
=
\frac{1}{\binom{n}{k}} \mydel{S = V}\mydel{T = U}
\]
as required.
\end{proof}

\label{sec:minor-orthogonal}
\section{The local theorem}\label{sec:local}

The goal of this section is to prove Theorem~\ref{thm:local}.
In fact we will prove something more general --- that Theorem~\ref{thm:local} 
holds when $R$ is {\em any} minor-orthogonal ensemble.
%
For the remainder of the section, we will assume the following setup:
we are given fixed integers $d,m$ and variables $x, y$ and the following 
matrices:
\begin{itemize}
\item $R \in \MM_{m, m}$, a uniformly distributed signed 
permutation matrix
\item $A_1, A_2 \in \MM_{d, m}$ and $B_1, B_2 \in \MM_{m, d}$ and $U \in 
\MM_{d, d}$ all of which are independent from $R$ and do not contain the 
variables $x$ and $y$ (but could contain other variables).
\end{itemize}

\begin{lemma} \label{lem:first}
Let $k, \ell \leq d$ be nonnegative integers, and let 
$S, T \in \binom{[d]}{k}$ and
$U, V \in \binom{[d]}{\ell}$.
Then
\[
\expect{R}{ [A_1 R B_1]_{S, T} [A_2 R^{T} B_2]_{U, V} } 
= 
\frac{1}{\binom{m}{k}} [A_1 B_2]_{S, V} [A_2 B_1]_{U, T} ~ \delta_{ \{ k = 
\ell \} }
\]
\end{lemma}
\begin{proof}
By (\ref{eq:mult}) we have 
\begin{align*}
&\expect{R}{[A_1RB_1]_{S, T} [A_2 R^{T} B_2]_{U, V} } \\
&= \frac{1}{\binom{m}{k}} 
\sum_{W, X \in \binom{[m]}{k}}\sum_{Y, Z \in \binom{[m]}{\ell}} 
[A_1]_{S, W} [B_1]_{X, T} [A_2]_{U, Y} [B_2]_{Z, V} 
~ \delta_{ \{ W = Z \} } \delta_{ \{ X = Y \} }
\\&= \frac{1}{\binom{m}{k}} 
\sum_{W, X \in \binom{[m]}{k}} [A_1]_{S, W} 
[B_1]_{X, T} [A_2]_{U, X} [B_2]_{W, V} ~ \delta_{ \{ k = \ell \} }
\\&= \frac{1}{\binom{m}{k}} [A_1 B_2]_{S, V} [A_2 B_1]_{U, T}~ \delta_{ \{ k = 
\ell \} }
\end{align*}

\end{proof}

\begin{lemma} \label{lem:second}
Let $k \leq d$ be nonnegative integers, and let $S, T \in \binom{[d]}{k}$ and 
consider the polynomials
\[
p(x, y) 
= \expect{R}{[ x A_1 R B_1 + y A_2 R^{T} B_2]_{S, T} } 
= \sum_{i} p_i x^i y^{k-i}
\]
and
\[
q(x, y) = [ x A_1B_2 + y A_2 B_1 ]_{S, T} = \sum_{i} q_i x^i y^{k-i}.
\]
Then
\[
p_i =
\delta_{ \{ i = k/2 \} } \frac{(-1)^i} {\binom{m}{i}} q_i 
\]
\end{lemma}
\begin{proof}
We have by (\ref{eq:add})
\[
[ x A_1 R B_1 + y A_2 R^{T} B_2]_{S, T} 
= \sum_i \sum_{W, X \subseteq \binom{[k]}{i}} 
(-1)^{\| W + X \|_1} x^{|W|} y^{|\overline{W}|}
[A_1 R B_1]_{W(S), X(T)} [A_2 R^{T} B_2]_{\overline{W}(S), \overline{X}(T)}
\]
so by Lemma~\ref{lem:first} we will have
\[
\expect{R}{ [ x A_1 R B_1 + y A_2 R^{T} B_2]_{S, T} } = 0
\]
whenever $|W| \neq |\overline{W}|$.  
To complete the lemma, it therefore remains to show that, for $k = 2t$, we have
$p_{t} = \frac{(-1)^{t}} {\binom{m}{t}} q_{t}$.

Using (\ref{eq:add}) and Lemma~\ref{lem:first} again, we have
\begin{align}
p_t
&= \frac{1}{\binom{m}{t}} \sum_{W, X \subseteq \binom{[k]}{t}} 
(-1)^{\| W + X \|_1} 
[A_1 B_2]_{W(S), \overline{X}(T)} [A_2 B_1]_{\overline{W}(S), X(T)} \notag
\\&= \frac{1}{\binom{m}{t}} \sum_{W, X \subseteq \binom{[k]}{t}} 
(-1)^{\| W + \overline{X} \|_1} 
[A_1 B_2]_{W(S), X(T)} [A_2 B_1]_{\overline{W}(S), \overline{X}(T)} 
\label{eq:pt}
\end{align}
whereas
\begin{equation}\label{eq:qt}
q_t 
= \sum_{W, X \subseteq \binom{[k]}{t}} 
(-1)^{\| W + X \|_1} 
[A_1 B_2]_{W(S), X(T)} [A_2 B_1]_{\overline{W}(S), X(T)}. 
\end{equation}
Now using the fact that 
\[
\| W \|_1 + \| \overline{W} \|_1 = \sum_{i=1}^k i = \binom{k+1}{2}
\]
we get
$
(-1)^{\| W + X \|_1} = (-1)^{\| W + \overline{X} \|_1} (-1)^{\binom{k+1}{2}}
$
and so (\ref{eq:pt}) and (\ref{eq:qt}) combine to give
\[
p_t = \frac{(-1)^{\binom{k+1}{2}}}{\binom{[m]}{t}} q_t.
\]
Hence it remains to show $(-1)^{\binom{k+1}{2}} = (-1)^t$.
However, (since $k = 2t$) we have 
\[
(-1)^{\binom{k+1}{2}} = (-1)^{t(2t+1)} = (-1)^{2 t^2 + t} = (-1)^t
\]
finishing the proof.
\end{proof}

\begin{corollary}
For any sets $S, T$ with $|S| = |T|$, 
\[
\expect{R}{[ x A_1 R B_2 + y A_2 R^{T} B_1]_{S, T} } 
= \sum_i (-1)^{i} \frac{(m-i)!}{m! i!} (\partial_w)^{i}(\partial_z)^{i} 
[ w x  A_1B_1 + y z A_2 B_2 ]_{S, T} \bigg|_{w = z = 0}
\]
\end{corollary}

\begin{corollary}\label{cor:Dat1}
For any sets $S, T$ with $|S| = |T|$, 
\[
\expect{R}{[(x A_1 + y A_2 R)(B_1 + R^{T} B_2)]_{S, T}}
= \sum_i (-1)^{i} \frac{(m-i)!}{m! i!} (\partial_w)^{i}(\partial_z)^{i} 
[ w x  A_1B_1 + y z A_2 B_2 ]_{S, T} \bigg|_{w = z = 1}
\]
\end{corollary}

We are now in a position to prove Theorem~\ref{thm:local}:

\begin{proof}[Proof of Theorem~\ref{thm:local}]
By (\ref{eq:add}), it suffices to show
\begin{equation}\label{eq:toshow}
\expect{R}{[(x A_1 + y A_2 R)(B_1 + R^{T} B_2)]_{S, T}} = L_m^{x, y} \left \{ 
[x A_1 B_1 + y A_2 B_2]_{S, T} \right \}
\end{equation}
for all $|S| = |T| = k$.  
For $k > m$, both sides of (\ref{eq:toshow}) are $0$, so we restrict to the 
case $k \leq m$.
Using Corollary~\ref{cor:Dat1}, (\ref{eq:toshow}) is equivalent to 
showing 
\[
L_m^{x, y} \left \{ 
[x A_1 B_1 + y A_2 B_2]_{S, T} \right \} =
\sum_i (-1)^{i} \frac{(m-i)!}{m! i!} (\partial_w)^{i}(\partial_z)^{i} 
[ w x  A_1B_2 + y z A_2 B_1 ]_{S, T} \bigg|_{w = z = 1}.
\]
Using (\ref{eq:add}) again, we have
\[
L_m^{x, y} \left \{ 
[x A_1 B_1 + y A_2 B_2]_{S, T} \right \}
=
\sum_{i} \sum_{W, X \subseteq \binom{[k]}{i}} 
(-1)^{\| W + X \|_1} 
L_m^{x, y}  \{ x^i y^{k-i} \}
[A_1 B_1]_{W(S), \overline{X}(T)} [A_2 B_2]_{\overline{W}(S), X(T)} 
\]
so the coefficient of $x^i y^{k-i}$ is 
\[
\sum_{W, X \subseteq \binom{[k]}{i}} 
(-1)^{\| W + X \|_1} 
\frac{(m-i)!(m-k+i)!}{m!(m-k)!}
[A_1 B_1]_{W(S), \overline{X}(T)} [A_2 B_2]_{\overline{W}(S), X(T)}.
\]
On the other hand, we have
\begin{align*}
&\sum_j (-1)^{j} \frac{(m-j)!}{m! j!} (\partial_w)^{i}(\partial_z)^{i} 
[ w x  A_1B_1 + y z A_2 B_2 ]_{S, T} \bigg|_{w = z = 1}
\\&=
\sum_{i, j} (-1)^{i} \frac{(m-j)!}{m! j!} (\partial_w)^{j}(\partial_z)^{j} 
\sum_{W, X \subseteq \binom{[k]}{j}} 
(-1)^{\| W + X \|_1} 
 (wx)^i (yz)^{k-i} 
[A_1 B_1]_{W(S), \overline{X}(T)} [A_2 B_2]_{\overline{W}(S), X(T)}
\\&=
\sum_{i, j} (-1)^{j} \frac{(m-j)!}{m! j!}
\sum_{W, X \subseteq \binom{[k]}{i}} 
(-1)^{\| W + X \|_1} 
\frac{i!}{(i-j)!}\frac{(k-i)!}{(k-i-j)!} x^i y^{k-i} 
[A_1 B_1]_{W(S), \overline{X}(T)} [A_2 B_2]_{\overline{W}(S), X(T)}
\end{align*}
so the coefficient of $x^i y^{k-i}$ is 
\[
\sum_{W, X \subseteq \binom{[k]}{i}} 
(-1)^{\| W + X \|_1} 
[A_1 B_1]_{W(S), \overline{X}(T)} [A_2 B_2]_{\overline{W}(S), X(T)}
\sum_j (-1)^{j} \frac{(m-j)!}{m! j!}
\frac{i!}{(i-j)!}\frac{(k-i)!}{(k-i-j)!}
\]
So it suffices to show 
\[
\sum_j (-1)^{j} \frac{(m-j)!}{m! j!}
\frac{i!}{(i-j)!}\frac{(k-i)!}{(k-i-j)!}
=
\frac{(m-i)!(m-k+i)!}{m!(m-k)!}.
\]
or, after substituting $i \leftarrow a$ and $b \leftarrow k - i$, to show
\[
\sum_j (-1)^{j} \frac{(m-j)!}{m! j!}
\frac{a!}{(a-j)!}\frac{b!}{(b-j)!}
=
\frac{(m-a)!(m-b)!}{m!(m-a-b)!}.
\]
whenever $a + b \leq m$ (our assumption).
However this follows directly from standard theorems on generalized binomials:
\begin{align*}
\sum_i (-1)^{i} \frac{(m-i)!}{m! i!}
\frac{a!}{(a-i)!}\frac{b!}{(b-i)!}
&= \frac{b!(m-b)!}{m!} \sum_i (-1)^i \binom{a}{i} \binom{m-i}{b-i} 
\\&= \frac{b!(m-b)!}{m!} (-1)^b \sum_i \binom{a}{i} \binom{-(m-b+1)}{b-i} 
\\&= \frac{b!(m-b)!}{m!} (-1)^b \binom{-(m-b-a+1)}{b} 
\\&= \frac{b!(m-b)!}{m!} \binom{m-a}{b} 
\\&= \frac{(m-a)!(m-b)!}{(m-a-b)! m!} 
\end{align*}
as required.
\end{proof}

\section{The global theorem}\label{sec:global}

The goal of this section is to prove Theorem~\ref{thm:global}.
Using the multilinearity of the mixed discriminant, one can show that Theorem~\ref{thm:global} is equivalent to the following theorem (note that $n! = D(I, \dots, I)$ and so this has the familiar form of zonal spherical polynomials --- see \cite{macd}).

\begin{theorem} \label{lem:separate}
Let $A_1, \dots, A_n, B_1, \dots, B_n, Q \in \MM_{n,n}$ where $Q$ is a 
uniformly distributed signed permutation matrix.
Then
\begin{equation}\label{eq}
\expect{Q}{D(A_1 Q B_1 Q^T, A_2 Q B_2 Q^T, \dots, A_n Q B_n Q^T)} = \frac{1}{n!} D(A_1, A_2, \dots, A_n)D(B_1, B_2, \dots, B_n).
\end{equation}
\end{theorem}
\begin{proof}
We start by noticing that, by the multilinearity of the mixed discriminant, it suffices to prove the theorem when the $A_i, B_i$ are the basis elements $\{ e_i e_j^T \}_{i, j=1}^n$.
So let $A_i = e_{w_i} e_{x_i}^T$ and $B_i = e_{y_i} e_{z_i}^T$ where $w_i, x_i, y_i, z_i \in [n]$.
Then for each $Q$ we have
\[
D(A_1 Q B_1 Q^T, \dots, A_n Q B_n Q^T)
= \mydet{Q} D(e_{w_1} e_{x_1}^T Q e_{y_1} e_{z_1}^T, \dots, e_{w_n} e_{x_n}^T Q e_{y_n} e_{z_n}^T)
\]
where each $e_{x_i}^T Q e_{y_i}$ is a scalar (and so can be factored out).
Hence 
\begin{align*}
D(A_1 Q B_1 Q^T, \dots, A_n Q B_n Q^T)
&= \mydet{Q} \left( \prod_i e_{x_i}^T Q e_{y_i} \right) D(e_{w_1} e_{z_1}^T, \dots, e_{w_n} e_{z_n}^T)
\\&= \mydet{Q} \left( \prod_i e_{x_i}^T Q e_{y_i} \right) \mydet{ e_{w_1}~\dots~e_{w_n}}\mydet{ e_{z_1}~\dots~e_{z_n}}
\end{align*}
On the other hand, 
\[
D(A_1, A_2, \dots, A_n) = \mydet{ e_{w_1}~\dots~e_{w_n}} \mydet{ e_{x_1}~\dots~e_{x_n}}
\]
and similarly for $B$, so we find that (\ref{eq}) is equivalent to showing 
\begin{equation}\label{eq2}
\expect{Q}{\mydet{Q} \prod_i e_{x_i}^T Q e_{y_i}} = \frac{1}{n!} \mydet{ e_{x_1}~\dots~e_{x_n}}\mydet{ e_{y_1}~\dots~e_{y_n}}.
\end{equation}
We now decompose\footnote{This is where we lose the generality of 
minor-orthogonal ensembles.} each $Q$ as $Q = P_\pi E_\chi$ where $P_\pi$ is a 
permutation matrix and $E_\chi$ is a diagonal matrix with diagonal entries 
$\chi_1, \dots, \chi_n \in \{ \pm 1 \}$.
Hence $\mydet{Q} = \left(\prod_i \chi_i \right) \mydet{P_\pi}$ and
\[
e_{x_i}^T Q e_{y_i} 
= e_{x_i}^T P_\pi E_\chi e_{y_i}
= \chi_{y_i} e_{x_i}^T e_{\pi(y_i)}
= \chi_{y_i} \delta_{ \{ x_i = \pi(y_i) \} }
\]
and so 
\[
\expect{Q}{\mydet{Q} \left( \prod_i e_{x_i}^T Q e_{y_i} \right)}
= \frac{1}{n!} \sum_{\pi} \mydet{P_\pi} \left(\prod_i \delta_{ \{ x_i = \pi(y_i) \}} \right) \expect{\chi_1, \dots, \chi_n}{ \left( \prod_i \chi_i \chi_{y_i} \right)}.
\]
Now it is easy to see that  
\[
\expect{\chi_1, \dots, \chi_n}{ \left( \prod_i \chi_i \chi_{y_i} \right)} = 1
\]
whenever the $y_i$ are distinct (that is, form a permutation of $[n]$) and $0$ otherwise.
For distinct $y_i$, it should then be clear that 
\[
\left(\prod_i \delta_{ \{ x_i = \pi(y_i) \}} \right) = 0
\]
unless the $x_i$ are also distinct.
Of course, this also holds for $\mydet{ e_{x_1}~\dots~e_{x_n}}$ and $\mydet{ e_{y_1}~\dots~e_{y_n}}$ and so (\ref{eq2}) is true whenever the $x_i$ or $y_i$ are not distinct (as both sides are $0$).  

Thus it remains to consider the case when $x_i = \sigma(i)$ for some $\sigma$ and $y_i = \tau(i)$ for some $\tau$. 
That is, we must show 
\begin{equation}\label{eq3}
 \sum_{\pi} \mydet{P_\pi} \left(\prod_i \delta_{ \{ \sigma(i) = \pi(\tau(i)) \}} \right) = \mydet{P_\sigma} \mydet{P_\tau}
\end{equation}
for all permutations $\tau$ and $\sigma$.
But now it is easy to see that 
$\left(\prod_i \delta_{ \{ \sigma(i) = \pi(\tau(i)) \}} \right) = 0$
for all permutations $\pi$ except for one: $\pi = \sigma \circ \tau^{-1}$.
Hence 
\[
 \sum_{\pi} \mydet{P_\pi} \left(\prod_i \delta_{ \{ \sigma(i) = \pi(\tau(i)) \}} \right) 
= \mydet{P_{\sigma \circ \tau^{-1}}} 
= \mydet{P_{\sigma}}\mydet{P_{\tau^{-1}}}
\]
where
\[
\mydet{P_{\tau^{-1}}} = \mydet{P_\tau^{-1}} = \mydet{P_\tau^T} = \mydet{P_\tau},
\]
proving (\ref{eq3}), which in turn proves the remaining (nonzero) cases of (\ref{eq2}), and therefore the theorem.

\end{proof}

\section{Applications}\label{sec:apps}

%
%

In this section, we list some direct applications of the main theorems.

\subsection{Permanents of Low Rank Matrices}

Our first application is an algorithm for computing permanents of low rank 
matrices that was originally discovered by Barvinok \cite{bar} using similar 
tools.
Barvinok's algorithm takes advantage of a well known connection between 
permanents and mixed discriminants: the permanent of a matrix $M \in \MM_{n, 
n}$ is the mixed discriminant $D(A_1, \dots, A_n)$ where each $A_i$ is a 
diagonal matrix with diagonal matching the $i$th column of $M$.

When working with diagonal matrices, the signed permutation matrices behave in a particularly nice way: the $\pm 1$ entries in $Q$ and $Q^{T}$ cancel, so one can reduce such formulas to an average over (unsigned) permutation matrices.
\begin{corollary}
Let $\{ A_i \}_{i=1}^k$ and $\{ B_i \}_{i=1}^k$ be $n \times n$ diagonal matrices and consider the polynomials
\[
p(x_1, \dots, x_k) = \mydet{ \sum_i x_i A_i }
\AND
q(x_1, \dots, x_k) = \mydet{ \sum_i x_i B_i }
\]
Then 
\[
\frac{1}{n!} \sum_{P \in \mathcal{P}_n} \mydet{ \sum_i x_i A_i P B_i P^T } = [ p \star q](x_1, \dots, x_k)
\]
\end{corollary}

Given a vector $v$, let $\diag(v)$ denote the diagonal matrix whose diagonal 
entries are $v$.
Note that if $A_i = \diag(a_i)$ and $B_i = \diag(b_i)$ for each $i$, then 
\[
\sum_{P \in \mathcal{P}_n} \mydet{ \sum_i x_i A_i P B_i P^T } = \perm\left[ \sum_i x_i a_i b_i^T \right].
\]

\vspace{0.25cm}
\noindent \textbf{Algorithm to find $\perm\left[ \sum_{i=1}^k a_i b_i^T \right]$} for $a_i, b_i \in \R^n$.
\begin{enumerate}
\item Form $A_i = \diag(a_i)$ and $B_i = \diag(b_i)$.
\item Compute $p(x_1, \dots, x_k) = \mydet{ \sum_{i=1}^k x_i A_i }$ and $q(x_1, \dots, x_k) = \mydet{ \sum_{i=1}^k x_i B_i }$
\item Compute $[p \star q](x_1, \dots, x_k)$
\item $\perm\left[ \sum_i a_i b_i^T \right] = n! [p \star q](1, \dots, 1)$.
\end{enumerate}

The complexity of this algorithm depends primarily on the number of terms in 
the polynomials $p$ and $q$, which (in general) will be the 
number of nonnegative integer solutions to the equation $\sum_{i=1}^k t_i = n$, 
which is known to be $\binom{n+k-1}{k-1}$ (see 
\cite{stanley}).

\subsection{Other convolutions}\label{sec:convolutions}

In this section, we show that the standard univariate convolutions defined in  
\cite{ffp} can each be derived from the main theorems.

\subsubsection{Additive convolution of eigenvalues}\label{sec:additive}
Given matrices $A, B \in \MM_{d, d}$ and polynomials 
\[
p(x) = \mydet{xI - A}
\AND
q(x) = \mydet{xI - B}
\]
the {\em additive convolution} of $p$ and $q$ can be written as 
\[
[p \boxplus q](x) = \expect{Q}{\mydet{xI - A - Q B Q^T}}
\]
where $Q$ can be chosen to be any minor-orthogonal ensemble (see \cite{dui}).
This can be achieved by setting 
\[
\hat{p}(x, y, z) = \mydet{x I + y A + z I}
\AND
\hat{q}(x, y, z) = \mydet{x I + y I + z B}
\]
and applying Theorem~\ref{thm:global} to get 
\[
\hat{p} \star \hat{q} = \expect{R}{\mydet{x I + y A + z R B R^T}}
\]
The formula for $[p \boxplus q]$ follows by setting $y = z = -1$.

%

\subsubsection{Multiplicative convolution of eigenvalues} 
\label{sec:multiplicative}
Given matrices $A, B \in \MM_{d, d}$ and polynomials 
\[
p(x) = \mydet{xI - A}
\AND
q(x) = \mydet{xI - B}
\]
the {\em multiplicative convolution} of $p$ and $q$ can be written as 
\[
[p \boxtimes q](x) = \expect{Q}{\mydet{xI - A Q B Q^T}}
\]
where $Q$ can be chosen to be any minor-orthogonal ensemble (see \cite{dui}).
Given matrices $A$ and $B$, this can be achieved by setting 
\[
\hat{p}(x, y) = \mydet{x I + y A}
\AND
\hat{q}(x, y) = \mydet{x I + y B}
\]
and applying Theorem~\ref{thm:global} to get 
\[
\hat{p} \star \hat{q} = \expect{Q}{\mydet{x I + y A Q B Q^T}}.
\]
The formula for $[p \boxtimes q]$ follows by setting $y = -1$.

\subsubsection{Additive convolution of singular values}\label{sec:rectangular}
Given matrices $A, B \in \MM_{d, n+d}$ and polynomials
\[
p(x) = \mydet{xI - AA^T}
\AND
q(x) = \mydet{xI - BB^T}
\]
the {\em rectangular additive convolution} of $p$ and $q$ can be written as 
\[
[p \boxplus_{d}^n q] 
= 
\expect{Q, R}{\mydet{xI - (A + Q B R)(A + Q B R)^T}}
\]
where $Q$ and $R$ can be chosen to be any (independent) minor-orthogonal 
ensembles of the appropriate size (see \cite{dui}).
This can be achieved by 
setting 
\[
\hat{r}(x, y, z) = \mydet{x I + ( y A + z Q B R) (A^T + R^T B^T Q^T }.
\]
Assuming $Q$ and $R$ are independent, we can do the expectation in $R$ using 
Theorem~\ref{thm:local} to get
\[
\expect{R}{\hat{r}(x, y, z)} 
= L_m^{y, z} \left\{ \mydet{x I + y AA^T + z Q BB^T Q^T} \right\}.
\]
and then we can compute the remaining expectation 
\[
\expect{Q}{\mydet{x I + y AA^T + z Q BB^T Q^T}}
\]
in terms of $p$ and $q$ using the method in Section~\ref{sec:additive}.

\subsubsection{Multiplicative convolution of non-Hermitian eigenvalues} 
\label{sec:nonHermitian}

For the purpose of studying the eigenvalues of non-Hermitian matrices, one 
could use the polynomial convolutions in the previous sections, but one quickly 
realizes that they do not hold as much information as one would like.
This is due in part to the fact that, unlike in the Hermitian case, there can 
be nontrivial relations between the left eigenvectors and right eigenvectors of 
a non-Hermitian matrix (we refer the interested reader to \cite{benno} where a 
multivariate theory is developed).
However it is well known that a non-Hermitian matrix $A$ can be written as $A = 
H + K$ where
\[
H = \frac{A + A^*}{2}
\AND
K = i\frac{A - A^*}{2}
\]
are both Hermitian.
One can then consider the multivariate polynomial 
\[
p(x, y, z) = \mydet{x I + y H + z K}
\]
for which an additive convolution follows easily from the Hermitian version in 
Section~\ref{sec:additive}.
The multiplicative version, however, is more complicated.
Given pairs of Hermitian matrices $(H_1, K_1)$ and $(H_2, K_2)$, one would like 
to ``convolve'' these matrices in a way that preserves the dichotomy between 
real and imaginary parts.
One such possibility would be the polynomial
\[
r(x, y, z) 
= \expect{Q}{\mydet{x I + y (H_1 Q H_2 Q^* - K_1 Q K_2 Q^*) + z(H_1 Q K_2 Q^* + 
K_1 Q H_2 Q^*) }}
\]
but it is not clear (a priori) that the coefficients of this polynomial are 
functions of the coefficients of the polynomials
\[
p_1(x) = \mydet{x I + y H_1 + z K_1}
\AND
p_2(x) = \mydet{x I + y H_2 + z K_2}.
\]
However it is easy to compute $r(x, y, z)$ using Theorem~\ref{thm:global}.
Letting 
\begin{align*}
q_1(x, a, b, c, d) &= \mydet{x I + a H_1 + b H_1 + c K_1 + d K_1} \\
q_2(x, a, b, c, d) &= \mydet{x I + a H_2 + b K_2 + c H_2 + d K_2}
\end{align*}
we have that 
\[
[q_1 \star q_2](x, a, b, c, d) = 
\expect{Q}{\mydet{x I + a H_1 Q H_2 Q^* + b H_1 Q K_2 Q^* + c K_1 Q H_2 Q^* + d 
K_1 Q K_2 Q^*}}
\]
and so $r(x, y, z) = [q_1 \star q_2](x, y, -y, z, z)$.

\section{An additive convolution for generalized singular values} 
\label{sec:gsvd}

There are three standard ensembles that one studies in random matrix 
theory: the Wigner ensemble, Wishart ensemble, and Jacobi ensemble 
\cite{forrester}.
All are alike in that they can be derived from matrices with independent 
Gaussian entries; the difference between them, as was first noted by Edelman 
\cite{edelman}, can be 
paralleled to different matrix decompositions.
The Wigner ensemble is Hermitian and the relevant distribution is the 
eigenvalues distribution.
The Wishart ensemble is often thought of as a Hermitian ensemble (with an 
eigenvalue distribution) but in some sense the more natural way to view it is 
as a distribution on singular values (which, as the first step in calculate 
them, you form a Hermitian matrix).
The Jacobi ensemble, in this ansatz, is most naturally viewed as a distribution 
on ``generalized singular values.'' 

One can attempt to explore this trichotomy further by studying how the
eigenvalues/singular values/generalized singular values of matrices behave with 
respect to more general matrix operations (and more general random 
matrices).
This is one of the motivations behind the polynomial convolutions mentioned in 
Section~\ref{sec:apps}: the convolution in Section~\ref{sec:additive} computes 
statistics concerning the eigenvalues of a unitarily invariant sum, whereas the 
convolution in Section~\ref{sec:rectangular} does similarly in the case of 
singular values.
The purpose of this section is to introduce a polynomial that can be used to 
study the final case: a unitarily invariant addition of generalized singular 
values.
While the previous two could be accomplished using univariate convolutions, it 
will become clear that this is not possible for the general singular value 
decomposition.
For those interested in other aspects of the GSVD should consult the references 
\cite{golub, 
gsvd}.

Before jumping into a discussion regarding the generalized singular value 
decomposition, it will be useful for us to recall the definition of the 
pseudo-inverse of a matrix.
Given any matrix $X \in \MM_{m, n}$ with rank $\rr$, the normal singular value 
decomposition of matrices allows us to write $X = U \Sigma V^{\trans}$ where
\begin{itemize}
\item $U \in \MM_{m, r}$ satisfies $U^\trans U = I$
\item $V \in \MM_{n, r}$ satisfies $V^\trans V = I$
\item $\Sigma \in \MM_{r, r}$ is diagonal and invertible. 
\end{itemize}
The {\em pseudo-inverse} of $X$ (written $X^{\inv} \in \MM_{n, m}$) is then 
defined to be $X^{\inv} = V \Sigma^{-1} U^{\trans}$.
The name ``pseudo-inverse'' comes from the fact that
\begin{itemize}
\item $X X^{\inv} \in \MM_{m, m}$ is the projection onto the column space of 
$V$, and
\item $X^{\inv} X \in \MM_{n, n}$ is the projection onto the column space of 
$U$.
\end{itemize}
So, in particular, if $n = \rr$ then $X^{\inv} X = I_{n}$ and if $m = n = \rr$ 
then $X$ is invertible and $X^{\inv} = X^{-1}$.

Now fix integers $\ss, \tt, \kk$ and let $M \in \MM_{(\ss + \tt), \kk}$ have 
rank $\rr$ and block structure
\[
M = \begin{bmatrix}
M_1 \\
M_2
\end{bmatrix}
\begin{array}{c} \ss \\ \tt \end{array}
\]
The {\em generalized singular value decomposition (GSVD)} provides a 
decomposition of $M_1$ and $M_2$ as 
\[
M_1 = U_1 C H
\AND
M_2 = U_2 S H
\]
where 
\begin{itemize}
\item $U_1 \in \MM_{\ss, \rr}$ and $U_2 \in \MM_{\tt, \rr}$ satisfy 
$U_1^\trans U_1 = U_2^\trans U_2 = I_\rr$
\item $C, S \in \MM_{\rr, \rr}$ are positive semidefinite diagonal matrices 
with $C^\trans C + S^\trans S = I$, and
\item $H \in \MM_{\rr, \kk}$ is some matrix with rank $\rr$.
\end{itemize}
In particular, the diagonal entries of $C$ and $S$ satisfy $c_i^2 + s_i^2 = 1$, 
and as such,  the matrices $C$ and $S$ are often referred to as {\em cosine} 
and {\em sine} matrices.
Note that when $M_2$ has rank $\rr$, the matrix $S$ will be invertible and then 
\[
M_1 M_2^{\inv} 
= (U_1 C H) (U_2 S H)^{\inv}
= U_1 C S^{-1} U_2
\]
will be the (usual) SVD of $M_1 M_2^{\inv}$, the reason for the nomenclature 
``generalized" SVD.

When $M$ has rank $\kk$, there is an easy way to find the generalized singular 
values without needing to form the entire decomposition.
Letting $W_1 = M_1^\trans M_1$ and $W_2 = M_2^\trans M_2$, the GSVD implies 
that 
\begin{equation}\label{eq:W}
W = (W_1 + W_2)^{-1/2} W_1 (W_1 + W_2)^{-1/2}
\end{equation}
is a positive semidefinite Hermitian matrix which is unitarily similar to 
$C^\trans C$ (all of whose eigenvalues are in the interval $[0, 1]$).
Thus generalized singular values can be found directly from the characteristic 
polynomial
\begin{equation}\label{eq:gsvdpoly}
\mydet{x I - (W_1 + W_2)^{-1/2} W_1 (W_1 + W_2)^{-1/2}}
= 
\mydet{(W_1 + W_2)^{-1}}\mydet{(x-1)W_1 + x W_2}
\end{equation}

So now assume we are given $M, N \in \MM_{(\ss + \tt), \kk}$ with block 
structure
\[
M = \begin{bmatrix}
M_1 \\
M_2
\end{bmatrix}
\begin{array}{c} \ss \\ \tt \end{array}
\AND
N = \begin{bmatrix}
N_1 \\
N_2
\end{bmatrix}
\begin{array}{c} \ss \\ \tt \end{array}
\]
and we form the random matrix
\[
P = \begin{bmatrix}
P_1 \\
P_2
\end{bmatrix}
= 
\begin{bmatrix}
M_1 + R_1 N_1 Q \\
N_2 + R_2 N_2 Q
\end{bmatrix}
\]
where $R_1, R_2, Q$ are independent signed permutation matrices of the 
appropriate sizes.
Then the natural question is: what (if anything) can we say about the 
generalized singular values of $P$ given the generalized singular values of $M$ 
and $N$.

By what we observed in (\ref{eq:gsvdpoly}), this means finding a correspondence 
between the polynomials
\[
\mydet{(x-1)M_1 + x M_2}
\AND
\mydet{(x-1)N_1 + x N_2}
\AND
\mydet{(x-1)P_1 + x P_2}
\]
The obvious first attempt is to consider the polynomials 
\[
p(x) = \mydet{x A_1^\trans A_1  + A_2^\trans A_2}.
\]
However when one starts to perturb $A_1$ and $A_2$ independently, one quickly 
realizes that simply knowing the generalized singular values are not enough --- 
information about $A_1$ and $A_2$ themselves is needed.
This motivates using a polynomial that keeps $A_1$ and $A_2$ independent (to 
some extent), which leads to the following definition:

Given $A \in \MM_{(\ss + \tt), \kk}$ with block 
structure
\[
A = \begin{bmatrix}
A_1 \\
A_2
\end{bmatrix}
\begin{array}{c} \ss \\ \tt \end{array}
\]
we define the {\em generalized singular value characteristic polynomial 
(GSVCP)} to be 
\begin{equation}\label{eq:trivariate}
p_A(x, y, z) = 
\mydet{x I + y A_1^\trans A_1 + z A_2^\trans A_2}
\end{equation}
The next theorem shows that (\ref{eq:trivariate}) defines a valid convolution 
--- that is, one can compute the GSVCP of a unitarily invariant sum of matrices 
from the GSVCPs of the summands.

\begin{theorem}\label{thm:gsvd}
Let $M, N, W \in \MM_{(\ss + \tt), \kk}$ with block structure
\[
M = \begin{bmatrix}
M_1 \\
M_2
\end{bmatrix}
\begin{array}{c} \ss \\ \tt \end{array}
\AND
N = \begin{bmatrix}
N_1 \\
N_2
\end{bmatrix}
\begin{array}{c} \ss \\ \tt \end{array}
\AND
W = \begin{bmatrix}
W_1 \\
W_2
\end{bmatrix}
= 
\begin{bmatrix}
M_1 + R_1 N_1 Q \\
M_2 + R_2 N_2 Q
\end{bmatrix}
\]
where $R_1, R_2, Q$ are independent, uniformly distributed, signed permutation 
matrices of the 
appropriate sizes and let
\begin{alignat*}{4} 
p_M(x, y, z)
&=
\mydet{x I + y M_1^\trans M_1 + z M_2^\trans M_2} 
&=& \sum_{j, k} 
\frac{x^{\kk-j-k}}{(\kk-j-k)!}
\frac{y^j}{(\ss-j)!} 
\frac{z^k}{(\tt-k)!}
p_{jk}
\\
p_N(x, y, z)
&=
\mydet{x I + y N_1^\trans N_1 + z N_2^\trans N_2} 
&=&
\sum_{j, k} 
\frac{x^{\kk-j-k}}{(\kk-j-k)!}
\frac{y^j}{(\ss-j)!} 
\frac{z^k}{(\tt-k)!}
q_{jk}
\\
p_W(x, y, z)
&=
\mydet{x I + y W_1^\trans W_1 + z W_2^\trans W_2} 
&=&
\sum_{j, k} 
\frac{x^{\kk-j-k}}{(\kk-j-k)!}
\frac{y^j}{(\ss-j)!} 
\frac{z^k}{(\tt-k)!}
r_{jk}
\end{alignat*}
be their GSVCPs, where each 
$r_{jk} = r_{jk}(R_1, R_2, Q)$ is a random variable.
Then 
\[
\expect{Q, R_1, R_2}{r_{jk}} = 
\begin{cases}
\frac{1}{\kk!\ss!\tt!}
\sum_{\beta = 0}^{j} \sum_{\delta = 0}^{k}
p_{\beta,\delta} q_{j-\beta, k-\delta}
& \text{for $j \leq \ss, k \leq \tt, j+k \leq \kk$} \\
0 & \text{otherwise}
\end{cases}
\]
\end{theorem}
\begin{proof}
We start by changing variables to match Theorem~\ref{thm:local}: 
let $f(x, s, t, u, v)$ denote the polynomial
\[
\expect{Q, R_1, R_2}{
\mydet{x I 
+ (s M_1 + t R_1 N_1 Q)^\trans (M_1 + R_1 N_1 Q) 
+ (u M_2 + v R_2 N_2 Q)^\trans (M_2 + R_2 N_2 Q)}}. 
\]
We now do the expectations separately, starting with $R_2$ and then $R_1$.
By Theorem~\ref{thm:local}, we get
\begin{align*}
f(x, s, t, u, v)
&= 
L_{\tt}^{u, v}
\expect{Q, R_1}{
\mydet{x I 
+ (s M_1 + t R_1 N_1 Q)^\trans (M_1 + R_1 N_1 Q) 
+ (u M_2^\trans M_2 + v Q^\trans N_2^\trans N_2 Q }}. 
\\&= 
L_{\ss}^{s, t}L_{\tt}^{u, v}  \expect{Q}{
\mydet{x I 
+ (s M_1^\trans M_1 + t Q^\trans N_1^\trans N_1 Q) 
+ (u M_2^\trans M_2 + v Q^\trans N_2^\trans N_2 Q }}. 
\end{align*}
By Theorem~\ref{thm:global} we have
\[
\expect{Q}{
\mydet{x I 
+ (s M_1^\trans M_1 + t Q^\trans N_1^\trans N_1 Q) 
+ (u M_2^\trans M_2 + v Q^\trans N_2^\trans N_2 Q }}
=
[g \star h](x, s, t, u, v)
\]
where
\[
g(x, s, t, u, v) 
= \mydet{ xI + s M_1^\trans M_1 + t I + u M_2^\trans M_2 + v I}
= p_M(x + t + v, s, u) 
\]
and
\[
h(x, s, t, u, v) 
= \mydet{ xI + s I + t N_1^\trans N_1 + u I + v N_2^\trans N_2}
= p_N(x + s + u,t, v)
\]
are each $\kk$-homogeneous polynomials.
We can now go in the reverse direction to compute $f(x, s, t, u, v)$ from the 
expansions in 
the hypothesis.
Firstly, we have
\begin{align*}
g(x, s, t, u, v)
&=
\sum_{j,k} 
\frac{(x + t + v)^{\kk-j-k}}{(\kk-j-k)!} 
\frac{s^j}{(\ss-j)!}
\frac{u^k}{(\tt-k)!} 
p_{jk}
\\&=
\sum_{j,k} \sum_{a,b} 
\frac{x^{\kk-j-k-a-b}}{(\kk-j-k-a-b)!}
\frac{t^a}{a!}
\frac{v^b}{b!} 
\frac{s^j}{(\ss-j)!}
\frac{u^k}{(\tt-k)!} 
p_{jk}
\\&= 
\sum_{\alpha + \beta + \gamma + \delta + \sigma = \kk}
\frac{x^\alpha s^\beta t^{\gamma} u^\delta 
v^{\sigma}}{\alpha!(\ss-\beta)!(\tt-\delta)!\gamma!\sigma!}
p_{\beta \delta}
\end{align*}
and similarly 
\begin{align*}
h(x, s, t, u, v) 
&=
\sum_{j,k} 
\frac{(x + s + u)^{\kk-j-k}}{(\kk-j-k)!} 
\frac{t^j}{(\ss-j)!}
\frac{v^k}{(\tt-k)!} 
p_{jk}
\\&=
\sum_{j,k} \sum_{a,b} 
\frac{x^{\kk-j-k-a-b}}{(\kk-j-k-a-b)!}
\frac{s^a}{a!}
\frac{u^b}{b!} 
\frac{t^j}{(\ss-j)!}
\frac{v^k}{(\tt-k)!} 
q_{jk}
\\&= 
\sum_{\alpha + \beta + \gamma + \delta + \sigma = \kk}
\frac{x^\alpha s^\beta t^{\gamma} u^\delta 
v^{\sigma}}{\alpha!(\ss-\gamma)!(\tt-\sigma)!\beta!\delta!}
q_{\gamma \sigma}
\end{align*}
Hence by definition of the star product, we have
\[
[g \star h](x, s, t, u, v)
= \frac{1}{\kk!}\sum_{\alpha, \beta, \gamma, \delta, \sigma}
\frac{x^\alpha s^\beta t^{\gamma} u^\delta v^{\sigma}} 
{\alpha!(\ss-\beta)!(\ss-\gamma)!(\tt-\delta)!(\tt-\sigma)! }
p_{\beta\delta}q_{\gamma \sigma}
\]
and so
\begin{align*}
f(x, s, t, u, v) 
&= L_{\ss}^{s, t}L_{\tt}^{u, v} \{ ~ [g \star h](x, s, t, u, v)~ \}
\\&= \frac{1}{\kk!\ss!\tt!}
\sum_{\substack{\alpha + \beta + \gamma + \delta + \sigma = \kk \\ \beta + 
\gamma \leq \ss \\ \delta + \sigma \leq \tt}}
\frac{x^\alpha s^\beta t^{\gamma} u^\delta v^{\sigma}}
{\alpha!(\ss-\beta-\gamma)!(\tt-\delta-\sigma)!}
p_{\beta \delta}
q_{\gamma \sigma}
\end{align*}
Putting all of this together, the polynomial we are interested in is
\begin{align*}
\expect{Q, R_1, R_2}{ p_W(x, y, z) } 
&= f(x, y, y, z, z)
\\&= 
\frac{1}{\kk!\ss!\tt!}
\sum_{\substack{\alpha + \beta + \gamma + \delta + \sigma = \kk \\ \beta + 
\gamma \leq \ss \\ \delta + \sigma \leq \tt}}
\frac{x^\alpha y^{\beta + \gamma} z^{\delta + \sigma}}
{\alpha!(\ss-\beta-\gamma)!(\tt-\delta-\sigma)!}
p_{\beta \delta}
q_{\gamma \sigma}
\\&= 
\frac{1}{\kk!\ss!\tt!}
\sum_{\substack{j\leq \ss \\ k \leq \tt \\ j+k \leq \kk}}
\frac{x^{\kk-j-k}}{(\kk-j-k)!} \frac{y^j}{ (\ss - j)! }\frac{z^k}
{(\tt-k)!}
\sum_{\beta = 0}^{j} \sum_{\delta = 0}^{k}
p_{\beta,\delta} q_{j-\beta, k-\delta}
\end{align*}
as required.
\end{proof}

Note that while Theorem~\ref{thm:gsvd} considers fixed matrices $M$ and $N$, 
one can easily extend it to random matrices using linearity of expectation.
We finish the section by observing that the convolution described by 
Theorem~\ref{thm:gsvd} has a remarkably simple form when the polynomials are 
expressed in the context of differential operators.

\begin{corollary}\label{cor:gsvd}
Let $p_M, p_N, p_W$ be the polynomials in Theorem~\ref{thm:gsvd} and let $P, Q$ 
be 
bivariate polynomials for which 
\[
y^{\ss} z^{\tt} p_M(x, 1/y, 1/z) 
= P(\partial_x \partial_y, \partial_x \partial_z) \{ 
x^\kk y^{\ss} z^{\tt} \}
\]
and
\[
y^{\ss} z^{\tt} p_N(x, 1/y, 1/z) 
= Q(\partial_x \partial_y, \partial_x \partial_z) \{ 
x^\kk y^{\ss} z^{\tt} \}.
\]
Then 
\[
\expect{R_1, R_2, Q}{y^{\ss} z^{\tt} p_W(x, 1/y, 1/z)} = 
P(\partial_x \partial_y, \partial_x \partial_z)Q(\partial_x \partial_y, 
\partial_x \partial_z) \{ 
x^{\kk} y^{\ss} z^{\tt} \}.
\]
\end{corollary}

Corollary~\ref{cor:gsvd} suggests that if $p_A(x, y, z)$ is the polynomial in 
(\ref{eq:trivariate}), then a more reasonable polynomial to consider would be 
\[
q_A(x, y, z) 
= y^{\ss} z^{\tt} p_A(x, 1/y, 1/z) = y^{\ss-\kk}z^{\tt-\kk} \mydet{x y z I + z 
A_1^\trans A_1 + y A_2^\trans A_2}.
\]
Another advantage to this alternative form is that there is a more direct 
matrix model that one can work with, as one can easily check that
\[
\det
\begin{bmatrix}
x I_{\kk}  & A_1^\trans & A_2^\trans \\
 A_1 & y I_{\ss}  & 0 \\
 A_2 & 0  & z I_{\tt}
\end{bmatrix}
= y^{\ss - \kk} z^{\tt-\kk} \mydet{x y z I - z A_1^\trans A_1 - y A_2^\trans 
A_2 }.
\]
Obviously the two are simple transformations from each other; we mention it 
because particular applications can be more well suited to one versus 
the other.

\section{Open Problems}\label{sec:conc}

The proof presented in Section~\ref{sec:local} applies more generally than 
Theorem~\ref{thm:local} in the respect that it holds for all minor-orthogonal 
ensembles. 
We suspect that Theorem~\ref{thm:global} has a similar generalization, but have 
not able to prove it.
This is not much of a hindrance when it comes to theoretical applications: the 
majority of the minor-orthogonal ensembles that one comes across in random 
matrix theory contain the signed permutation matrices as a subgroup and so 
Theorem~\ref{thm:global} can be extended them by the averaging argument in 
Corollary~\ref{cor:mo}.
The one notable situation where this is not the case is that of the uniform 
distribution over the standard representation of $S_{n+1}$ (what you get when 
you turn the collection of $(n+1) \times (n+1)$ permutation matrices into $n 
\times n$ matrices by projecting each one orthogonally to the constant vector).
This is (as far as the author knows) the minor-orthogonal ensemble with the 
smallest support and so is often useful for computational purposes.

Those familiar with the connection between polynomial convolutions and free 
probability (see, for example, \cite{ffmain}), might recognize the 
convolutions in Sections~\ref{sec:additive}, \ref{sec:multiplicative}, and 
\ref{sec:rectangular} as the ``finite free'' versions of the additive, 
multiplicative, and rectangular convolutions from classical free probability.
The operators from free probability have known closure properties (they map 
distributions on the real line to distributions on the real line) and so one 
might hope the same is true for the finite analogues.
This turns out to be true: the convolution in Section~\ref{sec:additive} maps 
Hermitian determinantal representations
to Hermitian determinantal representations and 
the ones in Sections~\ref{sec:multiplicative} and \ref{sec:rectangular} map 
positive semidefinite representations to semidefinite representations.
In the multivariate case, one can show (using a powerful theorem of Helton and 
Vinnikov \cite{hv}) that the convolution in Section~\ref{sec:gsvd} preserves 
positive semidefinite representations as well.
Continuing the analogy, the operator $\star$ would be the natural finite 
analogue of the {\em box product} $\boxed{\star}$ from free probability 
\cite{nica-speicher} 
and so one might hope that it, too, has some sort of closure property.
However this seems to be completely open.
It would therefore be both useful and interesting to understand the conditions 
under which the operations in this paper can be shown to preserve some (amy) 
class of polynomials.

While the expected characteristic polynomial of a random matrix gives you some 
information, it will (in general) not be enough to characterize the eigenvalue 
distribution of the underlying random matrix.
However, there is a natural way to ``assign'' an eigenvalue distribution to a 
convolution --- the one which is uniformly distributed over the roots of the 
polynomial (the fact that polynomial convolutions preserve real stability imply 
that this will be a valid distribution on the real line).
It is still not known exactly how the uniform-over-roots distributions derived 
from polynomial convolutions relate to the actual distributions of the 
underlying random matrices.
The one area that seems to show the most striking resemblances to this is that 
of free probability, which one can view as the study of the limiting 
distributions of 
random matrix theory as the dimension approaches $\infty$.
There is some speculation (mostly by this author) that polynomial convolutions 
represent the limiting distributions of random matrix theory as some other 
parameter (usually referred to as $\beta$) approaches $\infty$.
There is some evidence supporting this idea \cite{vadim}, but in many cases it 
is not clear how to even define such a limit formally.
Understanding this relationship better, however, is certainly an interesting 
open problem (and one that remains fairly wide open).

\section{Acknowledgements}\label{sec:thanks}

This paper was a direct consequence of the IPAM program in Quantitative Linear Algebra.
Essentially all of the results here were discovered as a result of discussions that 
took place during this program.
The author specifically thanks Benno Mirabelli, who (among other things) 
pointed out the relationship between the convolution in Theorem~\ref{global} 
and the box convolution in free probability.

\end{document}